\documentclass{article}
\usepackage[utf8]{inputenc}
\usepackage[T1]{fontenc}
\usepackage{amsthm}
\usepackage{amssymb}
\usepackage{amsmath}
\usepackage{amsfonts}
\usepackage[activate={true,nocompatibility}]{microtype}
\usepackage{amstext}
\usepackage{enumitem}
\usepackage{textcomp}
\usepackage{hyphenat}
\usepackage{csquotes}
\usepackage{hyperref}
\usepackage{wasysym}
\usepackage[margin=1.5cm, bottom=2cm]{geometry}
\usepackage{faktor}

\usepackage[
    style       = alphabetic,
    sorting     = ynt,
    backend     = bibtex,
    autolang    = other]{biblatex}
\bibliography{Coarse.bib}

\newtheorem{theorem}{Theorem}[section]
\newtheorem{corollary}{Corollary}[theorem]

\newtheorem{definition}{Definition}[section]

\newtheorem{lemma}{Lemma}[section]
\newtheorem{proposition}[lemma]{Proposition}

\newtheorem{exmpl}{Example}[section]

\numberwithin{equation}{section}

\title{Bornological Metrics on Groups}
\author{A.\,A.\,Arutyunov\footnote{The results of Section~2.2 were
obtained by A.\,Arutyunov. The work of A.\,Arutyunov was supported
by the Russian Science Foundation under grant no.~25-11-00018,
\url{https://rscf.ru/en/project/25-11-00018/}, and performed at
Lomonosov Moscow State University.}
\and A.\,V.\,Perelygin}

\begin{document}
\maketitle

\section{Introduction}

Let $G$ be a group. In many mathematical contexts it is natural to
regard $G$ as a metric space, with the connection between the metric
and the algebraic structure expressed by left invariance: a metric
$\rho\colon G\times G\to\mathbb{R}$ is \emph{left-invariant} if
\[
    \rho(x,y)=\rho(gx,gy),\qquad\forall\,g,x,y\in G.
\]

In~\cite{smith2006} J.\,Smith proved that a proper left-invariant
metric is unique up to coarse equivalence, where properness means
that every ball of finite radius contains only finitely many elements
of the group.

In the present paper we give a description of left-invariant
\emph{improper} metrics, using the construction of bornological
groups introduced by R.\,Tessera and J.\,Winkel
in~\cite{Tessera2022CoarseFP}.  In the course of studying bornological
groups we also clarify some of their properties; in particular, we
establish a metrizability criterion for a bornological group.

It turns out that, from the point of view of coarse geometry, there
is no difference between the classical left-invariance condition and
its coarse analogue, namely the condition we call
\emph{bornologicity} of a metric: a metric $\rho$ is bornological if
for every $C>0$ there exists a constant $S_C>0$ such that for any
pair $x,y\in G$ with $\rho(x,y)<C$ one has
\[
    \rho(gx,gy)<S_C,\qquad\forall\,g\in G.
\]

We show that each coarse-equivalence class of bornological metrics is
determined by a bornology (see Definition~\ref{def-bornologstr}).

Metrics subject to similar conditions appear in various contexts.
Besides~\cite{smith2006}, we mention the application of this
construction to the study of uniform Roe algebras~\cite{Chung2021,
Chung2024-wc}, as well as its use in the study of asymptotic
dimension~\cite{Bell2008-dz,Bell2008-qg-1}.  In~\cite{artigue2026}
the opposite conditions are imposed on the metric, leading to the
study of expansivity metrics.

\bigskip

The main results of the paper are as follows.

\medskip

\textbf{Proposition~\ref{свойство метрики}.}
\textit{Let $\rho$ be a metric on a group $G$.  Then $\rho$ is
bornological if and only if there exists a bornology $\mathfrak{B}$
such that the left coarse structure $\mathcal{E}_L$ coincides with
the bounded coarse structure $\mathcal{E}_{\rho}$.  The bornology
$\mathfrak{B}$ consists of all sets of finite diameter.}

\medskip

Moreover, within the coarse-equivalence class of bornological metrics
there is a canonical representative, namely a left-invariant metric
on the group.

\medskip

\textbf{Theorem~\ref{th-leftinvmetric}.}
\textit{In every coarse-equivalence class of bornological metrics on
a group $G$ there exists a left-invariant metric.}

\medskip

Coarsely equivalent metrics can be identified by the classes of sets
that are bounded with respect to them.

\medskip

\textbf{Theorem~\ref{th-coarseequivstruct}.}
\textit{Let $\mathcal{E}_1$ and $\mathcal{E}_2$ be coarsely
connected, strongly $G$-invariant left coarse structures with the
same families of coarsely bounded sets,
$\mathfrak{B}(\mathcal{E}_1)=\mathfrak{B}(\mathcal{E}_2)$.  Then
the coarse spaces $(G,\mathcal{E}_1)$ and $(G,\mathcal{E}_2)$ are
coarsely equivalent.}

\medskip

It is natural to ask which bornologies arise from metrics.  The
answer takes the form natural to coarse geometry, via the notion of
countable generation.  For the precise statement we need the concept
of a strongly $G$-invariant coarse structure
(see Definition~\ref{def-g-inv}): a coarse structure $\mathcal{E}$
is strongly $G$-invariant if for every controlled set $E\in\mathcal{E}$
and every decomposition $E=\bigsqcup_{i=1}^{\infty}E_i$ with
$E_i\in\mathcal{E}$, and for every sequence $g=\{g_i\mid g_i\in G\}$,
one has
\[
    gE:=\bigcup_{i=1}^{\infty}g_i E_i\in\mathcal{E},
\]
where $g_i E_i:=\{(g_i x,\,g_i y)\mid(x,y)\in E_i\}$.

\medskip

\textbf{Theorem~\ref{th-metrizable}.}
\textit{\begin{enumerate}[label=\normalfont(\arabic*)]
    \item If the coarse structure $\mathcal{E}$ on a group $G$ is
    countably generated and strongly $G$-invariant, then
    $\mathfrak{B}(\mathcal{E})$ is a metrizable bornology.
    \item If the bornology $\mathfrak{B}(\mathcal{E})$ on a
    bornological group $G$ is metrizable, then the coarse structure
    $\mathcal{E}$ is countably generated and strongly $G$-invariant
    on the left.
\end{enumerate}}

\medskip

The connection between metrizability and the existence of
left-invariant metrics was previously discussed in~\cite{ma2021endslargescalegroups}
(see Section~7); however, a complete description of improper metric
structures in terms of bornologies appears to be new.

\medskip

In the final part of the paper we study various classes of
non-equivalent improper bornological metrics.  In particular, the
following criterion is useful for constructing such examples.

\medskip

\textbf{Theorem~\ref{th-critnonequiv}.}
\textit{Let $G$ be an infinite finitely generated group and let $d$
be a proper left-invariant metric on $G$.  If $N$ is a normal
subgroup of finite index and $\rho_{\faktor{G}{N}}$ is the word
metric on the quotient group, then $(G,d)$ is not coarsely equivalent
to $\bigl(\faktor{G}{N},\rho_{\faktor{G}{N}}\bigr)$.}

\subsection{Bornological and left-invariant metrics on groups}

We begin by fixing the definition and notation for a metric on an
abstract space~$X$.

\begin{definition}
  A map $\rho\colon X\times X\to\mathbb{R}_{\geq 0}$ is called a
  \emph{metric} if:
  \begin{enumerate}
      \item $\rho(x,y)=0\ \Leftrightarrow\ x=y$;
      \item $\rho(x,y)=\rho(y,x)$ for all $x,y$;
      \item $\rho(x,z)\leqslant\rho(x,y)+\rho(y,z)$ for all $x,y,z$.
  \end{enumerate}
\end{definition}

We are primarily interested in metrics on groups, which are
conveniently defined via norms. Let $G$ be a countable group; we
denote its identity element by~$e_G$.

\begin{definition}\label{def-bounded-set}
  A map $d\colon G\to[0,+\infty)$ is called a \emph{norm} on the
  group~$G$ if:
  \begin{itemize}
      \item $d(g)=0$ if and only if $g=e_G$;
      \item $d(g^{-1})=d(g)$ for every $g\in G$;
      \item $d(gh)\leq d(g)+d(h)$ for all $g,h\in G$.
  \end{itemize}
\end{definition}

A norm $\|\cdot\|$ induces a metric $\rho_{\|\cdot\|}$ on $G$ by
\[
    \rho_{\|\cdot\|}(g,h):=\|g^{-1}h\|.
\]
The induced metric is left-invariant: $\rho(ag,ah)=\rho(g,h)$ for
every $a\in G$.

\bigskip

We also give the definition of a proper norm.

\begin{definition}
  A norm $d$ on $G$ is called \emph{proper} if for every $R>0$ there
  are only finitely many $g\in G$ with $d(g)\leq R$. A metric induced
  by a proper norm is called proper.
\end{definition}

A finitely generated group is naturally equipped with a \emph{word
norm} and the induced metric with respect to a fixed generating set
$\mathcal{X}$. The word norm $\|g\|$ of an element $g\in G$ is
defined as the minimal length of a word in the generators
$\mathcal{X}$ representing~$g$. This metric is both proper and
left-invariant.

A more general statement for countable (not necessarily finitely
generated) groups is proved in~\cite[Proposition~1]{smith2006}. The
notion of coarse equivalence is defined below
(see Definition~\ref{def-coarse-equiv-metrics}).

\begin{theorem}[J.\,Smith]\label{Smith}
  On a countable group, any two proper left-invariant metrics are
  coarsely equivalent.
\end{theorem}

We now drop the requirement that metrics (and norms) be proper and
study all left-invariant metrics. The left-invariance condition itself
can also be relaxed.

\begin{definition}
  A metric $\rho$ on a group $G$ is called \emph{left bornological}
  if for every $C>0$ there exists a constant $S_C>0$ such that for
  any pair $x,y\in G$ with $\rho(x,y)<C$ one has
  \[
      \rho(gx,gy)<S_C,\qquad\forall\,g\in G.
  \]
\end{definition}

One can analogously define a right bornological metric, but we do not
consider this notion here. Accordingly, we henceforth refer to left
bornological metrics simply as bornological metrics.

\subsection{Coarse structures}

To define coarse equivalence of metrics we recall the standard
definitions from coarse geometry that we shall use. As in most other
standard definitions in coarse geometry, we follow the terminology
of~\cite{Roe2003LecturesOC}.

\begin{definition}
  A pair $(X,\mathcal{E})$ is called a \emph{coarse space}, and the
  family of subsets $\mathcal{E}\subset 2^{X\times X}$ is called a
  \emph{coarse structure} on~$X$, if $\mathcal{E}$ satisfies the
  following conditions:
  \begin{enumerate}
      \item If $E\in\mathcal{E}$ and $E_1\subset E$, then
            $E_1\in\mathcal{E}$.
      \item If $E_1,\dots,E_n\in\mathcal{E}$, then
            $\bigcup_{i=1}^{n}E_i\in\mathcal{E}$.
      \item If $E\in\mathcal{E}$, then $E^{-1}\in\mathcal{E}$, where
            \[
                E^{-1}=\{(y,x)\mid(x,y)\in E\}.
            \]
      \item If $E_1,E_2\in\mathcal{E}$, then
            $E_1\circ E_2\in\mathcal{E}$, where
            \[
                E_1\circ E_2=\{(x,z)\mid\exists\,y\colon
                (x,y)\in E_1\text{ and }(y,z)\in E_2\}.
            \]
      \item $\operatorname{diag}_{X\times X}=\{(x,x)\mid x\in X\}
            \in\mathcal{E}$.
  \end{enumerate}
\end{definition}

The central example of a coarse structure is the bounded coarse
structure generated by a metric space $(X,\rho)$.

\begin{definition}\label{def-boundedcoarsestructurt}
  The \emph{bounded coarse structure} $\mathcal{E}_\rho$ on a metric
  space $(X,\rho)$ is the family of subsets of $X\times X$ defined by
  \begin{equation}
      E\in\mathcal{E}_\rho\ \Leftrightarrow\
      \bigl\{\exists\,C>0\colon\forall\,(x,x')\in E,\quad
      \rho(x,x')<C\bigr\}.
  \end{equation}
\end{definition}

That $\mathcal{E}_\rho$ is indeed a coarse structure follows directly
from the definitions.

\bigskip

Closely related to the notion of a coarse structure is the notion of
a coarsely bounded set.

\begin{definition}
  Let $(X,\mathcal{E})$ be a coarse space. A subset $B\subset X$ is
  called \emph{coarsely bounded} if $B\times B\in\mathcal{E}$. We
  denote the family of coarsely bounded sets by~$\mathfrak{B}(\mathcal{E})$.
\end{definition}

For the bounded coarse structure of
Definition~\ref{def-boundedcoarsestructurt}, coarse boundedness
coincides with ordinary metric boundedness.

We record the following lemma for completeness; it will be used later.

\begin{lemma}\label{lemma-bounded}
  The following conditions for a coarse space $(X,\mathcal{E})$ are
  equivalent:
  \begin{enumerate}
      \item $B\times B\in\mathcal{E}$;
      \item there exists a point $x\in X$ such that
            $B\times\{x\}\in\mathcal{E}$.
  \end{enumerate}
\end{lemma}

\begin{proof}
\begin{enumerate}
    \item By the first axiom of a coarse structure, $B\times
    B\in\mathcal{E}$ implies $B\times\{x\}\in\mathcal{E}$ for every
    $x\in B$.
    \item Suppose $B\times\{x\}\in\mathcal{E}$. By the axioms of a
    coarse structure, $\{x\}\times B=(B\times\{x\})^{-1}\in\mathcal{E}$.
    Applying the composition operation to $B\times\{x\}$ and
    $\{x\}\times B$ gives $B\times B\in\mathcal{E}$.\qedhere
\end{enumerate}
\end{proof}

\bigskip

We now give the general definition of coarse equivalence. Let $X$
and~$Y$ be coarse spaces.

\begin{definition}
  A map $f\colon X\to Y$ is called \emph{proper} if for every
  coarsely bounded set $B\subset Y$ the preimage $f^{-1}(B)$ is also
  bounded.
\end{definition}

\begin{definition}
  A map $f\colon X\to Y$ is called \emph{bornologous} if for every
  controlled set~$E$ the set $(f\times f)(E)$ is also controlled.
\end{definition}

\begin{definition}
  A map $f\colon X\to Y$ is called a \emph{coarse map} if it is both
  proper and bornologous.
\end{definition}

Let $M$ be an arbitrary set (not necessarily equipped with a coarse
structure) and let $(X,\mathcal{E})$ be a coarse space.

\begin{definition}
  Two maps $f,f'\colon M\to X$ are called \emph{close} if the set
  \[
      \{(f(m),f'(m))\mid m\in M\}
  \]
  is controlled.
\end{definition}

When $X$ is a metric space with the bounded coarse structure
(Definition~\ref{def-boundedcoarsestructurt}), closeness is
equivalent to the existence of a uniform bound on the distances
between images: for some constant $C\geq 0$,
\[
    \rho(f(m),f'(m))\leq C.
\]

We can now give the definition of coarse equivalence.

\begin{definition}\label{def-coarse-equiv}
  Coarse spaces $X$ and $Y$ are \emph{coarsely equivalent}
  $(X\stackrel{c}{\sim}Y)$ if there exist coarse maps $f\colon X\to Y$
  and $g\colon Y\to X$ such that $f\circ g$ is close to
  $\operatorname{id}_Y$ and $g\circ f$ is close to
  $\operatorname{id}_X$.
\end{definition}

For metric spaces this specializes as follows.

\begin{definition}\label{def-coarse-equiv-metrics}
  Two metrics $\rho_1$ and $\rho_2$ on a group $G$ are called
  \emph{coarsely equivalent} if the coarse space $(G,\mathcal{E}_{\rho_1})$
  is coarsely equivalent to $(G,\mathcal{E}_{\rho_2})$.
\end{definition}

A natural example of coarsely equivalent metrics is provided by
proportional metrics; however, coarsely equivalent metrics need not
be proportional.

\section{Bornological groups}

We follow the definitions of~\cite{Tessera2022CoarseFP}. Let $G$ be
a group; we denote by $2^G$ the set of all subsets of~$G$.

\begin{definition}\label{def-bornologstr}
  A \emph{bornology} on $G$ is a family of subsets
  $\mathfrak{B}\subset 2^G$ satisfying the following axioms:
  \begin{enumerate}
      \item For every $g\in G$ the singleton $\{g\}$ belongs to
            $\mathfrak{B}$.
      \item If $A\subset B\subset G$ and $B\in\mathfrak{B}$, then
            $A\in\mathfrak{B}$.
      \item If $A\in\mathfrak{B}$, then $A^{-1}\in\mathfrak{B}$.
      \item If $A,B\in\mathfrak{B}$, then $A\cup B\in\mathfrak{B}$
            and $A\cdot B\in\mathfrak{B}$.
  \end{enumerate}
\end{definition}

A group equipped with a bornology is naturally called a
\emph{bornological group}. The elements of the bornology, as will
become clear below, are naturally called \emph{bounded sets}.

\begin{definition}
  The pair $(G,\mathfrak{B})$ is called a \emph{bornological group}.
  A set $A\in\mathfrak{B}$ is called \emph{bornologically bounded}
  (or simply \emph{bounded} when no confusion can arise).
\end{definition}

\begin{exmpl}
  Let $G$ be a group. The family of all finite subsets of $G$ forms a
  bornology~$\mathfrak{B}_{\min}$.
\end{exmpl}

In this bornology the bounded sets are precisely the finite sets. At
the same time, it is easy to see from the axioms that finite subsets
belong to every bornology, so $\mathfrak{B}_{\min}$ is the minimal
bornological structure on~$G$.

The maximal bornology is defined analogously.

\begin{exmpl}\label{ex-Bfull}
  The family of all subsets $\mathfrak{B}_{\mathrm{full}}:=2^G$
  forms a bornology. In this bornology every subset of $G$ is
  bounded.
\end{exmpl}

\subsection{Left and right bornological coarse structures}

A bornology $\mathfrak{B}$ naturally gives rise to two coarse
structures, which we now define.

\begin{definition}
  \begin{enumerate}
      \item The \emph{left bornological coarse structure} is defined
            by
            \begin{equation}
                E\in\mathcal{E}_L\ \Leftrightarrow\
                \{x^{-1}y\mid(x,y)\in E\}\in\mathfrak{B}.
            \end{equation}
      \item The \emph{right bornological coarse structure} is defined
            by
            \begin{equation}
                E\in\mathcal{E}_R\ \Leftrightarrow\
                \{xy^{-1}\mid(x,y)\in E\}\in\mathfrak{B}.
            \end{equation}
  \end{enumerate}
\end{definition}

The family of coarsely bounded sets for both the left and right
bornological coarse structures coincides with the bornology.

\begin{proposition}
  For $\mathcal{E}_L$ (respectively $\mathcal{E}_R$) one has
  $\mathfrak{B}(\mathcal{E}_L)=\mathfrak{B}$, that is, the coarsely
  bounded sets coincide with the bornologically bounded sets.
\end{proposition}

\begin{proof}
  Suppose $B\in\mathfrak{B}$. By Lemma~\ref{lemma-bounded} it
  suffices to show that $B\times\{g\}\in\mathcal{E}_L$ for some
  $g\in G$. Indeed, $\{x^{-1}g\mid x\in B\}=B^{-1}\cdot g\in\mathfrak{B}$
  by the axioms of a bornology, so $B\times\{g\}\in\mathcal{E}_L$ by
  definition. Hence $\mathfrak{B}\subset\mathfrak{B}(\mathcal{E}_L)$.

  Conversely, suppose $B\in\mathfrak{B}(\mathcal{E}_L)$. Then there
  exists $E\in\mathcal{E}_L$ such that $B=E_{\{g\}}:=\{x\mid
  (x,g)\in E\}$ for some $g\in G$. Since $E\in\mathcal{E}_L$, the
  set $A=\{x^{-1}y\mid(x,y)\in E\}\in\mathfrak{B}$. By definition
  $B\subset\{g\}\cdot A^{-1}\in\mathfrak{B}$, so
  $\mathfrak{B}(\mathcal{E}_L)\subset\mathfrak{B}$.
\end{proof}

As expected, the coarse spaces $(G,\mathcal{E}_L)$ and
$(G,\mathcal{E}_R)$ are coarsely equivalent.

\begin{proposition}\label{правое и левое}
  Let $G$ be a group with a bornology. Then
  \[
      (G,\mathfrak{B},\mathcal{E}_L)\stackrel{c}{\sim}
      (G,\mathfrak{B},\mathcal{E}_R).
  \]
\end{proposition}

\begin{proof}
  Consider the map $\theta\colon x\mapsto x^{-1}$. We show that it
  is a coarse equivalence.
  \begin{itemize}
      \item \emph{Properness.} The preimage of a bounded set is
            bounded, since the bornology is closed under taking
            inverses.
      \item \emph{Bornologicity.} Let $E=\{(x,y)\}\in\mathcal{E}_L$,
            so that $A_E=\{x^{-1}y\mid(x,y)\in E\}\in\mathfrak{B}$.
            Then $(\theta\times\theta)(E)=\{(x^{-1},y^{-1})\mid
            (x,y)\in E\}$. Since
            $A_{(\theta\times\theta)(E)}=\{x^{-1}y\}=A_E\in\mathfrak{B}$,
            we have $(\theta\times\theta)(E)\in\mathcal{E}_R$.
      \item Since $\theta$ is a bijection, the compositions
            $\theta\circ\theta^{-1}$ and $\theta^{-1}\circ\theta$
            are equal (and hence close) to the identity.\qedhere
  \end{itemize}
\end{proof}

\subsection{Coarse structures associated with a bornology}

We define a class of coarse structures compatible with the group
structure.

\begin{definition}\label{def-g-inv}
  A coarse structure $\mathcal{E}$ is called \emph{strongly
  $G$-invariant on the left} if for every $E\in\mathcal{E}$, every
  decomposition $E=\bigsqcup_{i\in I}E_i$ with $E_i\in\mathcal{E}$
  (where $I$ is an index set and some $E_i$ may be empty), and every
  family $\bar{g}=\{g_h\mid h\in G\}$, the set
  \[
      \bar{g}E:=\bigcup_{h\in G}g_h E_h
  \]
  is controlled, where $g_h E_h:=\{(g_h x,g_h y)\mid(x,y)\in E_h\}$.
\end{definition}

\begin{exmpl}
  The left bornological coarse structure $\mathcal{E}_L$ is strongly
  $G$-invariant on the left.
\end{exmpl}

\begin{definition}
  A coarse structure $\mathcal{E}$ on a space $X$ is called
  \emph{coarsely connected} if for any $x,x'\in X$ there exists a
  controlled set $E\in\mathcal{E}$ such that $(x,x')\in E$.
\end{definition}

Coarse connectedness can be informally understood as the
commensurability of any two points. For the bounded coarse structure
generated by a metric, coarse connectedness is equivalent to ordinary
metric connectedness.

\begin{proposition}\label{главное предложние}
  If $\mathcal{E}$ is strongly $G$-invariant on the left and coarsely
  connected, then the family
  \[
      \mathfrak{B}:=\bigl\{B_E=\{x^{-1}y\mid(x,y)\in E\}
      \mid E\in\mathcal{E}\bigr\}
  \]
  is a bornology and $\mathcal{E}=\mathcal{E}_L$.
\end{proposition}

\begin{proof}
  We verify the axioms of a bornology.
  \begin{enumerate}
      \item Singletons belong to $\mathfrak{B}$.
      \item Closure under taking subsets, finite unions, and inverses
            follows from the axioms of a coarse structure.
      \item Closure under products follows from strong
            $G$-invariance. Let
            \[
                A=\{x^{-1}y\mid(x,y)\in E_A\},\qquad
                B=\{a^{-1}b\mid(a,b)\in E_B\}.
            \]
            Then
            \[
                A\cdot B=\{x^{-1}ya^{-1}b\mid(x,y)\in E_A,\,
                (a,b)\in E_B\}.
            \]
            Using strong $G$-invariance applied to $E_A$, define the
            controlled sets
            \[
                E:=\{(xy^{-1},e)\},\qquad E':=\{(e,a^{-1}b)\}.
            \]
            Applying Lemma~\ref{lemma-bounded} to $E\circ E'$ gives
            the required result.
  \end{enumerate}
  The inclusion $\mathcal{E}_L\subset\mathcal{E}$ is clear. The
  construction of the bornology gives the reverse inclusion
  $\mathcal{E}\subset\mathcal{E}_L$.
\end{proof}

\begin{corollary}\label{следствие про сильную g-инвариантность}
  If $\mathcal{E}$ is strongly $G$-invariant on the left and coarsely
  connected, then $\mathfrak{B}(\mathcal{E})$ is a bornology
  consisting of the sets described in
  Proposition~\ref{главное предложние}.
\end{corollary}

Considering strongly $G$-invariant coarse structures allows us to
strengthen Proposition~\ref{правое и левое}. We show that two such
coarse structures on~$G$ are coarsely equivalent whenever their
bornologies coincide. The proof uses the following result
from~\cite[Theorem~6.3]{brodskiy2006coarsestructuresgroupactions}.

\begin{theorem}[Brodskiy--Dydak--Mitra]\label{th-BDM}
  Let $\alpha_i\colon G_i\times X\to X$, $i=1,2$, be two commuting
  left actions of groups $G_i$ on a set~$X$. Suppose the coarse
  structures $\mathcal{E}_i$, $i=1,2$, have the same families of
  bounded sets and each action $\alpha_i$ is coarse with respect to
  $\mathcal{E}_i$. Then:
  \begin{enumerate}
      \item $G_1$ and $G_2$ are coarsely equivalent;
      \item $(X,\mathcal{E}_1)$ is coarsely equivalent to
            $(X,\mathcal{E}_2)$.
  \end{enumerate}
\end{theorem}

We now define a coarse action, following~\cite{brodskiy2006coarsestructuresgroupactions}.

\begin{definition}
  An action $\alpha\colon G\times X\to X$ of a group on a coarse
  space $(X,\mathcal{E})$ is called \emph{coarse} if:
  \begin{enumerate}
      \item For each $x\in X$ the map $\phi_x\colon G\to G\cdot x$
            is coarsely proper, i.e., the preimage of every bounded
            set is bounded.
      \item The action is cobounded: for some bounded set $U$ one has
            $X=G\cdot U$.
      \item The action is uniformly bornological: for every controlled
            set $E$ there exists a controlled set $E'$ such that
            $(g\cdot x,g\cdot y)\in E'$ for all $(x,y)\in E$ and
            $g\in G$.
  \end{enumerate}
\end{definition}

\begin{theorem}\label{th-coarseequivstruct}
  Let $\mathcal{E}_1$ and $\mathcal{E}_2$ be coarsely connected,
  strongly $G$-invariant left coarse structures with the same families
  of coarsely bounded sets,
  $\mathfrak{B}(\mathcal{E}_1)=\mathfrak{B}(\mathcal{E}_2)$. Then
  the coarse spaces $(G,\mathcal{E}_1)$ and $(G,\mathcal{E}_2)$ are
  coarsely equivalent.
\end{theorem}

\begin{proof}
  We show that the natural left action of a bornological group on
  itself, $G\times G\to G$, is coarse.
  \begin{itemize}
      \item \emph{Coboundedness.} The action is transitive and free,
            so any single point serves as~$U$.
      \item \emph{Uniform bornologicity.} This follows immediately
            from strong $G$-invariance applied to $E$ with the
            constant sequence $\bar{g}=(g,g,\ldots)$.
      \item \emph{Coarse properness.} This follows from the
            definition of a bornological group.
  \end{itemize}
  It remains to apply Theorem~\ref{th-BDM}
  from~\cite{brodskiy2006coarsestructuresgroupactions} to the left
  translation action of $G$ on itself.
\end{proof}

\subsection{Examples}

We give several examples of bornological groups and the resulting
left coarse structures. We begin with various bornological structures
on the additive group of integers. The maximal bornology, as in
Example~\ref{ex-Bfull}, consists of all subsets of the group.

\begin{exmpl}
  The group $\mathbb{Z}$ with the bornology
  $\mathfrak{B}_{\mathrm{full}}=2^{\mathbb{Z}}$ has left bornological coarse
  structure $\mathcal{E}_L=2^{\mathbb{Z}\times\mathbb{Z}}$.
\end{exmpl}

The intersection of any family of bornologies is again a bornology,
so the minimal bornology (with respect to inclusion) is the family of
finite subsets.

\begin{exmpl}
  The group $\mathbb{Z}$ with the minimal bornology
  $\mathfrak{B}_{\min}$ has left bornological coarse structure $\mathcal{E}_L$
  coinciding with $\mathcal{E}_\rho$, where $\rho$ is the metric
  $\rho(x,y)=|x-y|$.
  \begin{proof}
    Minimality of the bornology gives $\mathcal{E}_L\subset\mathcal{E}_\rho$.

    For the reverse inclusion, let $E\in\mathcal{E}_\rho$. Then the
    set $\{x^{-1}y\mid(x,y)\in E\}=\{y-x\mid(x,y)\in E\}$ is
    finite, since $|x-y|\leq C$ for some $C$ implies that $y-x$
    takes only finitely many values. Hence $E\in\mathcal{E}_L$.
  \end{proof}
\end{exmpl}

We turn to a more interesting example: the Heisenberg group, which is
the group of upper unitriangular integer matrices of size~$3\times 3$.

Equip the Heisenberg group $\mathcal{H}$ with the metric
\[
    \rho(A,B)=\max_{i,j}|a_{ij}-b_{ij}|,
\]
and denote the corresponding coarse structure by $\mathcal{E}_\rho$.
As the bornology we take the family of sets that are bounded with
respect to this metric; these are precisely the finite sets.

\begin{exmpl}\label{пример для Гейзенберга}
  The left bornological coarse structure of $\mathcal{H}$ does not
  coincide with the bounded coarse structure $\mathcal{E}_\rho$:
  $\mathcal{E}_L\neq\mathcal{E}_\rho$.
\end{exmpl}

  \begin{proof}
    Let
    \[
        B_n=\begin{pmatrix}1&n+1&1\\0&1&1\\0&0&1\end{pmatrix},\qquad
        A_n=\begin{pmatrix}1&n&1\\0&1&0\\0&0&1\end{pmatrix}.
    \]
    Then $\rho(A_n,B_n)=1$, but $\rho(B_n^{-1}A_n,E)=n+1$. Hence the
    set $E=\bigcup_{n=1}^{\infty}(B_n,A_n)$ is controlled in
    $\mathcal{E}_\rho$ but not in $\mathcal{E}_L$, since
    $\{B_n^{-1}A_n\}$ is unbounded. Therefore
    $\mathcal{E}_\rho\neq\mathcal{E}_L$.
  \end{proof}
  Thus, although the bornology and the bounded coarse structure are
  both generated by the same metric, the two coarse structures do not
  coincide.

Hence bornological coarse structures need not coincide in general.
For the case where the bornology consists of finite sets it is known
(see~\cite[Proposition~2.1]{brodskiy2006coarsestructuresgroupactions})
that the left and right coarse structures coincide if and only if $G$
is an FC-group, i.e., every conjugacy class in $G$ is finite.

\bigskip

Before giving another useful example of a bornological structure, we
recall the definition of a pseudometric.

\begin{definition}
  A map $\rho\colon X\times X\to\mathbb{R}_{\geq 0}$ is called a
  \emph{pseudometric} if:
  \begin{enumerate}
      \item $\rho(x,x)=0$ for all $x$;
      \item $\rho(x,y)=\rho(y,x)$ for all $x,y$;
      \item $\rho(x,z)\leqslant\rho(x,y)+\rho(y,z)$ for all $x,y,z$.
  \end{enumerate}
\end{definition}

A pseudometric differs from a metric in that $\rho(x,y)$ may equal
zero for $x\neq y$.

\begin{exmpl}
  Consider the same group as in Example~\ref{пример для Гейзенберга},
  but equipped with the pseudometric $\rho(A,B)=|a_{12}-b_{12}|$.
\end{exmpl}

\begin{proof}
  Call a set bounded if it has finite diameter. Axioms~1--3 of a
  bornology are obvious; for axiom~4 it remains to verify closure
  under products, which follows from the matrix multiplication
  formula. Hence we indeed obtain a bornology.

  For the pseudometric defined above, the bounded coarse structure
  $\mathcal{E}_\rho$ is well defined. The matrix multiplication
  formula gives $\rho(E,A^{-1}\cdot B)=\rho(A,B)$ for this
  pseudometric, whence $\mathcal{E}_L=\mathcal{E}_\rho$.
\end{proof}

This example is of interest because it shows that pseudometrics can
generate bornologies.

\subsection{Metrizability of bornological coarse structures}

Recall that metrizability of a coarse structure means, informally,
that there exists a metric (not uniquely determined) generating it as
a bounded coarse structure.

\begin{definition}
  A coarse structure $\mathcal{E}$ on a set $X$ is \emph{metrizable}
  if there exists a metric $\rho$ on $X$ such that
  $\mathcal{E}=\mathcal{E}_\rho$.
\end{definition}

The following metrizability criterion is proved
in~\cite[Theorem~2.55]{Roe2003LecturesOC}.

\begin{definition}
    Let us say that a family of sets $\{S_i\mid S_i\subset 2^{X\times X}\}_{i\in I}$, where $I$ is a set of indices, generates a coarse structure $\mathcal{E}$ on the set $X$ if: 
    \begin{itemize}
        \item $S_i \subset \mathcal{E} $ for all $i \in I$;
        \item Structure $\mathcal{E}$ has the minimum inclusion among all structures containing the entire set $S_i$.
    \end{itemize}
\end{definition}

In that case, we will write $\mathcal{E}=\langle S_i\rangle$.

\begin{theorem}\label{roe-metrizable}
  A coarse structure $\mathcal{E}$ is metrizable if and only if it is
  countably generated.
\end{theorem}

\begin{proposition}\label{свойство метрики}
  Let $\rho$ be a metric on a group $G$. Then $\rho$ is bornological
  if and only if there exists a bornology $\mathfrak{B}$ such that
  the left coarse structure $\mathcal{E}_L$ coincides with the
  bounded coarse structure $\mathcal{E}_\rho$. The bornology
  $\mathfrak{B}$ consists of all sets of finite diameter.
\end{proposition}

\begin{proof}
  Suppose $\rho$ is bornological. Consider the family $\mathfrak{B}$
  consisting of all subsets of $G$ of the form
  \[
      \{x^{-1}y\mid(x,y)\in E\},\qquad E\in\mathcal{E}_\rho.
  \]
  \begin{itemize}
      \item \emph{$\mathfrak{B}$ is a bornology.}
      \begin{enumerate}
          \item Singletons belong to $\mathfrak{B}$.
          \item A subset of a bounded set is bounded, since
                $E'\subset E$ and $E\in\mathcal{E}_L$ imply
                $E'\in\mathcal{E}_L$.
          \item Closure under inverses is verified analogously.
          \item Closure under products: consider the set
                \[
                    \{x^{-1}ya^{-1}b\mid(x,y)\in E,\,(a,b)\in E'\}.
                \]
                To prove boundedness it suffices to show this set
                arises from a controlled set. Consider pairs of the
                form $(y^{-1}x,a^{-1}b)$. Then
                \[
                    \rho(y^{-1}x,a^{-1}b)\leq\rho(y^{-1}x,e)
                    +\rho(e,a^{-1}b)<C
                \]
                by bornologicity of $\rho$, which gives the required
                boundedness.
          \item Closure under finite unions follows from the axioms
                of a coarse structure.
      \end{enumerate}

      \item \emph{$\mathcal{E}_L=\mathcal{E}_\rho$.} The inclusion
            $\mathcal{E}_L\subset\mathcal{E}_\rho$ follows from the
            construction. For the reverse inclusion, let
            $E\in\mathcal{E}_\rho$; then by construction
            $\{x^{-1}y\mid(x,y)\in E\}\in\mathfrak{B}$, so
            $E\in\mathcal{E}_L$.

      \item \emph{Necessity.} Suppose there exists a bornology
            $\mathfrak{B}$ with $\mathcal{E}_L=\mathcal{E}_\rho$,
            and assume for contradiction that $\rho$ is not
            bornological. Then there exists $C>0$ such that for every
            $S_C>0$ one can find $(x,y)$ with $\rho(x,y)<C$ and
            $g\in G$ with $\rho(gx,gy)\geq S_C$. Consider the
            controlled set
            \[
                E_C=\{(x,y)\in X\times X\mid\rho(x,y)\leq C\}.
            \]
            Then the set $\{(gx,gy)\}$ belongs to $\mathcal{E}_L$
            but not to $\mathcal{E}_\rho$, contradicting
            $\mathcal{E}_L=\mathcal{E}_\rho$.\qedhere
  \end{itemize}
\end{proof}

We now begin to investigate metrizability conditions. First we
establish the following proposition.

\begin{proposition}\label{конечный диаметр, если гр структура метризуема}
  Let $(G,\mathfrak{B},\mathcal{E}_L)$ be a bornological group. If
  $\mathcal{E}_L$ is metrizable by a metric $\rho$, then for every
  subset $B\subset G$,
  \[
      B\in\mathfrak{B}\ \Leftrightarrow\
      \operatorname{diam}_\rho(B)<+\infty.
  \]
\end{proposition}

\begin{proof}
  Suppose $B\in\mathfrak{B}$. Consider the set
  $E_B=\{(e,x)\mid x\in B\}$. Since
  $\{e^{-1}x\mid x\in B\}=B\in\mathfrak{B}$, we have
  $E_B\in\mathcal{E}_L$. By metrizability,
  \[
      \exists\,C>0\colon\forall\,(e,x)\in E_B,\quad\rho(e,x)<C,
  \]
  so $\operatorname{diam}_\rho(B)\leq 2C$.

  \bigskip

  Conversely, suppose $B\subset G$ and
  $\operatorname{diam}_\rho(B)<+\infty$. Set
  $C=\operatorname{diam}_\rho(B)$ and fix $x\in B$. Then
  $\rho(x,y)<C$ for all $y\in B$, so
  \[
      \rho(e,y)\leq\rho(e,x)+C\qquad\forall\,y\in B.
  \]
  Hence $E=\{(e,y)\mid y\in B\}\in\mathcal{E}_L$, and therefore
  $B\in\mathfrak{B}$.
\end{proof}

This proposition leads naturally to the following definition.

\begin{definition}
  A bornology $\mathfrak{B}$ is called \emph{metrizable} if there
  exists a metric $\rho$ such that $B\in\mathfrak{B}$ if and only if
  $\operatorname{diam}_\rho(B)<\infty$.
\end{definition}

Given a bornological metric $\rho$ on $G$, define
\[
    \rho^+(x,y):=\sup_{g\in G}\rho(gx,gy).
\]
As we show below, this is a left-invariant metric coarsely equivalent
to~$\rho$.

\begin{theorem}\label{th-leftinvmetric}
  In every coarse-equivalence class of bornological metrics on a
  group $G$ there exists a left-invariant metric.
\end{theorem}

\begin{proof}
  \begin{itemize}
      \item \emph{$\rho^+$ is a metric.} Symmetry, positive
            definiteness, and non-degeneracy are clear. The triangle
            inequality:
            \begin{equation*}
              \begin{split}
                \rho^+(x,z)
                  &=\sup_{g\in G}\rho(gx,gz)
                   \leq\sup_{g\in G}\bigl(\rho(gx,gy)+\rho(gy,gz)\bigr)\\
                  &\leq\sup_{g\in G}\rho(gx,gy)
                        +\sup_{g\in G}\rho(gy,gz)
                   =\rho^+(x,y)+\rho^+(y,z).
              \end{split}
            \end{equation*}
      \item \emph{Left invariance of $\rho^+$} follows immediately
            from the definition.
      \item \emph{$(G,\mathcal{E}_\rho)\stackrel{c}{\sim}
            (G,\mathcal{E}_{\rho^+})$.} It suffices to show that the
            identity map
            $\operatorname{id}\colon(G,\mathcal{E}_\rho)\to
            (G,\mathcal{E}_{\rho^+})$ is a coarse equivalence.
            Bornologicity of $\operatorname{id}$ in both directions
            follows from the bornologicity of $\rho$ and $\rho^+$.
            Properness follows from the fact that both bornologies
            consist of sets of finite diameter
            (Proposition~\ref{свойство метрики}). Closeness of the
            compositions to the identity is clear.\qedhere
  \end{itemize}
\end{proof}

It is now clear that bornological metrics are in bijective
correspondence with metrizable bornologies: metrizability of a
bornology is equivalent to metrizability of the left bornological
coarse structure coinciding with $\mathcal{E}_\rho$.

As we now show, metrizability of the bornological coarse structure is
equivalent to the existence of a countable basis for the bornology.

\begin{definition}\label{def-basisbor}
  A bornology $\mathfrak{B}$ has a \emph{countable basis} if there
  exists a sequence $\{B_i\mid B_i\in\mathfrak{B}\}_{i=1}^{\infty}$,
  called a basis of the bornology, such that for every
  $B\in\mathfrak{B}$ there exists $N$ with $B\subset\bigcup_{j=1}^{N}B_j$.
\end{definition}

\begin{proposition}\label{связь}
  For a group $G$ the following are equivalent:
  \begin{enumerate}
      \item $\mathcal{E}_L$ is countably generated;
      \item $\mathcal{E}_R$ is countably generated;
      \item $\mathfrak{B}$ has a countable basis.
  \end{enumerate}
\end{proposition}

\begin{proof}
  Define the map $\Omega\colon X\times X\to X$ by $\Omega(x,y)=x^{-1}y$.
  \begin{itemize}
      \item[$3\Rightarrow 1$.] Let $\{B_i\}_{i=1}^{\infty}$ be a
      countable basis of $\mathfrak{B}$. Set
      \[
          E_n:=\{(x,y)\mid\Omega(x,y)\in B_n\}.
      \]
      Clearly $\langle E_n\rangle\subset\mathcal{E}_L$. Let
      $E\in\mathcal{E}_L$; then $\Omega(E)\subset\bigcup_{i=1}^{k}B_i$
      for some $k$, so $E\subset\bigcup_{i=1}^{k}E_i$. Hence
      $\mathcal{E}_L\subset\langle E_n\rangle$.

      \item[$1\Rightarrow 3$.] Since $\mathcal{E}_L$ is countably
      generated it is metrizable (Theorem~\ref{roe-metrizable}), so
      Proposition~\ref{конечный диаметр, если гр структура метризуема}
      applies: the sets $E_n(e)=\{g\mid\rho(e,g)<n+1\}$ form a
      countable basis of $\mathfrak{B}$.

      \item[$1\Leftrightarrow 2$.] By
      Proposition~\ref{правое и левое} the coarse spaces are
      equivalent, hence their coarse structures are metrizable
      simultaneously. By Theorem~\ref{roe-metrizable} this is
      equivalent to countable generation.\qedhere
  \end{itemize}
\end{proof}

The equivalence between a bornology having a countable basis and
being metrizable was noted in~\cite[Proposition~7.1]{ma2021endslargescalegroups},
where an explicit formula for the corresponding metric is also given.

\begin{corollary}\label{критерий метризуемости борнологии}
  A bornology $\mathfrak{B}$ has a countable basis if and only if it
  is metrizable.
\end{corollary}

\begin{proof}
  If $\mathfrak{B}$ is metrizable then it has a countable basis (see
  the end of the proof of Proposition~\ref{связь}). Conversely, if
  $\mathfrak{B}$ has a countable basis, then by
  Proposition~\ref{связь} the structure $\mathcal{E}_L$ is countably
  generated, hence metrizable, and
  Proposition~\ref{конечный диаметр, если гр структура метризуема}
  gives metrizability of~$\mathfrak{B}$.
\end{proof}

For finitely generated groups the word metric corresponds to the
minimal bornology.

\begin{proposition}\label{для конечно порожденных групп}
  If $G$ is finitely generated, then $\mathfrak{B}_{\min}$ is
  generated by the word metric.
\end{proposition}

\begin{proof}
  The bornology $\mathfrak{B}_{\min}$ consists of finite sets. Every
  finite set has finite diameter with respect to the word metric.
  Conversely, every set of finite diameter with respect to the word
  metric in a finitely generated group is finite, by definition of
  the word metric.
\end{proof}

For countably generated groups the construction of the corresponding
metric is more involved; this was addressed in~\cite{smith2006}.

\bigskip

Combining Propositions~\ref{конечный диаметр, если гр структура метризуема},
\ref{связь}, and~\ref{свойство метрики}, we obtain the following
description of the relation between metrizability and coarse
structures.

\begin{theorem}\label{th-metrizable}
  \begin{enumerate}
      \item If the coarse structure $\mathcal{E}$ on a group $G$ is
            countably generated and strongly $G$-invariant on the
            left, then $\mathfrak{B}(\mathcal{E})$ is a metrizable
            bornology.
      \item If the bornology $\mathfrak{B}(\mathcal{E})$ on a
            bornological group $G$ is metrizable, then $\mathcal{E}$
            is countably generated and strongly $G$-invariant on the
            left.
  \end{enumerate}
\end{theorem}

\begin{proof}
  \begin{enumerate}
      \item Since $\mathcal{E}$ is countably generated it is
            metrizable. Since $\mathcal{E}$ is strongly
            $G$-invariant, the corresponding metric is left
            bornological. By Corollary~\ref{следствие про сильную
            g-инвариантность}, $\mathfrak{B}(\mathcal{E})$ is a
            bornology, and by Proposition~\ref{конечный диаметр,
            если гр структура метризуема} it is metrizable.
      \item If $\mathfrak{B}(\mathcal{E})$ is metrizable then it has
            a countable basis. By Proposition~\ref{связь},
            $\mathcal{E}$ is countably generated, hence metrizable.
            By Proposition~\ref{свойство метрики} the corresponding
            metric is left bornological, so $\mathcal{E}$ is strongly
            $G$-invariant on the left.\qedhere
  \end{enumerate}
\end{proof}

\subsection{Equivalence of bornological groups}

We carry over the coarse equivalence relation defined above for
coarse spaces to bornological groups.

\begin{definition}
  A bornological group $(G,\mathfrak{B})$ is \emph{coarsely
  equivalent} to a bornological group $(H,\mathfrak{R})$, written
  $(G,\mathfrak{B})\stackrel{c}{\sim}(H,\mathfrak{R})$, if $G$ and
  $H$ are coarsely equivalent as coarse spaces with structures
  $\mathcal{E}_{L_{\mathfrak{B}}}$ and $\mathcal{E}_{L_{\mathfrak{R}}}$
  respectively.
\end{definition}

For bornologies with a countable basis this definition is equivalent
to coarse equivalence of the corresponding metric spaces.

\begin{lemma}\label{ПОЛЕЗНАЯ ЛЕММА}
  Let $f$ be a surjective coarse map from a bornological group
  $(G,\mathfrak{B})$ with countable basis to a bornological group
  $(H,\mathfrak{R})$ with countable basis. Then
  $f(\mathfrak{B})=\mathfrak{R}$.
\end{lemma}

As will be clear from the proof, it suffices to require surjectivity
only onto bounded sets.

\begin{proof}
  Since $f\colon(G,\mathfrak{B})\to(H,\mathfrak{R})$ is a coarse
  map, every controlled set $E\in\mathcal{E}_{L_{\mathfrak{B}}}$
  maps to a controlled set $(f\times f)(E)\in\mathcal{E}_{L_{\mathfrak{R}}}$.
  Since both bornologies have countable bases, by
  Proposition~\ref{связь} the left coarse structures are countably
  generated, hence metrizable. The condition that $f\times f$ maps
  controlled sets to controlled sets means
  \[
      \forall\,x,y\in G\;\bigl(\rho_1(x,y)<C\Rightarrow
      \rho_2(f(x),f(y))<S_C\bigr).
  \]
  Hence sets of finite diameter map to sets of finite diameter; by
  Theorem~\ref{конечный диаметр, если гр структура метризуема} this
  gives $f(\mathfrak{B})\subset\mathfrak{R}$. Since $f$ is proper,
  $f^{-1}(\mathfrak{R})\subset\mathfrak{B}$, and surjectivity gives
  $\mathfrak{R}\subset f(\mathfrak{B})$. Hence
  $\mathfrak{R}=f(\mathfrak{B})$.
\end{proof}

We now give a criterion for coarse equivalence of metrizable
bornological structures on a given group. Let $\mathfrak{B}_1$ and
$\mathfrak{B}_2$ be countably generated bornologies with countable
bases $\{B_i^1\}_{i=1}^{\infty}$ and $\{B_i^2\}_{i=1}^{\infty}$
respectively.

\begin{proposition}
  \begin{enumerate}
      \item \emph{(Necessary condition for surjections.)} For
            bornologies $\mathfrak{B}_1$ and $\mathfrak{B}_2$ with
            countable bases $\{B_i^1\}_{i=1}^{\infty}$ and
            $\{B_i^2\}_{i=1}^{\infty}$ to be coarsely equivalent via
            surjections, i.e., $(G,\mathfrak{B}_1)\stackrel{c}{\sim}
            (G,\mathfrak{B}_2)$, it is necessary that every coarse
            map sends the basis of $\mathfrak{B}_1$ to a basis of
            $\mathfrak{B}_2$.

      \item \emph{(Sufficient condition.)} Bornologies $\mathfrak{B}_1$
            and $\mathfrak{B}_2$ on $G$ with countable bases are
            coarsely equivalent, $(G,\mathfrak{B}_1)\stackrel{c}{\sim}
            (G,\mathfrak{B}_2)$, if there exists a bijection
            $f\colon G\to G$ such that the image of the basis of
            $\mathfrak{B}_1$ is a basis of $\mathfrak{B}_2$.
  \end{enumerate}
\end{proposition}

We recall that the definition of a basis of a bornology does not
require minimality in any sense.

\begin{proof}
  \begin{enumerate}
      \item By Lemma~\ref{ПОЛЕЗНАЯ ЛЕММА}, every coarse map $f$
            satisfies $\mathfrak{B}_2=f(\mathfrak{B}_1)$. Hence for
            every $B\in\mathfrak{B}_2$ there exists $k$ such that
            \[
                f^{-1}(B)\subset\bigcup_{i=1}^{k}B_i^1,
            \]
            and therefore $B\subset\bigcup_{i=1}^{k}f(B_i^1)$.
            Thus $\{f(B_i^1)\}_{i=1}^{\infty}$ is a basis of
            $\mathfrak{B}_2$.

      \item We show that the bijection $f\colon(G,\mathfrak{B}_1)\to
            (G,\mathfrak{B}_2)$ and its inverse $f^{-1}$ form a
            coarse equivalence. The preimage of a bounded set is
            bounded since $f$ is a bijection sending the basis of
            $\mathfrak{B}_1$ to the basis of $\mathfrak{B}_2$.
            Similarly, bounded sets map to bounded sets, so for every
            $E\in\mathcal{E}_{L_{\mathfrak{B}_1}}$ we have
            $(f\times f)(E)\in\mathcal{E}_{L_{\mathfrak{B}_2}}$;
            hence $f$ is bornological and therefore a coarse map. The
            same argument applies to $f^{-1}$. Therefore
            $(G,\mathfrak{B}_1)\stackrel{c}{\sim}(G,\mathfrak{B}_2)$.\qedhere
  \end{enumerate}
\end{proof}

Combining Theorem~\ref{th-leftinvmetric} and Theorem~\ref{Smith}, we
obtain that on a countable group all bornologies with a countable
basis that generate bornological metrics coarsely equivalent to some
proper left-invariant metric are coarsely equivalent to one another.

\section{Examples of improper metrics}

To construct examples of improper left-invariant metrics we need to
understand which metrics defined on subgroups can be lifted to the
whole group, and how.

We reformulate the notion of a basis of a bornology
(Definition~\ref{def-basisbor}) in a form more convenient for our
purposes.

\begin{definition}\label{def-<bor>}
  Let $G$ be a group and let $A\subset G$. We denote by $\langle A\rangle$
  the minimal bornology on $G$ containing the set~$A$.
\end{definition}

By definition, $\langle A\rangle$ contains all singletons, all sets
of the form $gA$ and $Ag$, all finite powers of $A$, and all their
finite products, finite unions, and arbitrary subsets. An explicit
description of this bornology is rather cumbersome; it follows from
the definition that the minimal bornology containing $A$ coincides
with the intersection of all bornologies containing~$A$.

Definition~\ref{def-<bor>} is related to the notion of a basis
(Definition~\ref{def-basisbor}) as follows.

\begin{proposition}
  Let $G$ be a countable group. If a bornology $\mathfrak{B}$ is
  countably generated by sets $B_i$, that is,
  $\mathfrak{B}=\langle B_i\rangle$, then it has a countable basis.
\end{proposition}

\begin{proof}
  All finite combinations of the sets $B_i$ and elements of $G$
  obtainable by group multiplication, shifts, and finite unions form
  a countable basis, since there are countably many sets $B_i$ and
  the group itself is countable.
\end{proof}

\medskip

Let $G$ be a countable group and let $N$ be a normal subgroup of~$G$.
Equip the quotient group $\faktor{G}{N}$ with the word metric
$\rho_{\faktor{G}{N}}$, and define a pseudometric on $G$ by
\[
    \rho_G(a,b):=\rho_{\faktor{G}{N}}(\pi(a),\pi(b)),
\]
where $\pi\colon G\to\faktor{G}{N}$ is the canonical quotient
homomorphism.

\begin{proposition}\label{реализуемость метрики на подгруппе}
  If $G$ is finitely generated, then
  \begin{equation}\label{eq-factormetric}
      (G,\langle N\rangle)\stackrel{c}{\sim}
      \Bigl(\faktor{G}{N},\rho_{\faktor{G}{N}}\Bigr).
  \end{equation}
\end{proposition}

\begin{proof}
  \begin{itemize}
      \item Since $N$ is a normal subgroup, the bornology
            $\langle N\rangle$ consists of finite unions of singletons
            and cosets.

      \item We show that $\langle N\rangle$ coincides with the
            bornology $\mathfrak{B}_{\rho_G}$ constructed from the
            pseudometric $\rho_G$ (see Theorem~\ref{th-leftinvmetric}).
            From the description of $\langle N\rangle$ it is clear
            that $\langle N\rangle\subset\mathfrak{B}_{\rho_G}$, since
            every set in $\langle N\rangle$ has finite diameter. Now
            let $A\subset G$ satisfy $\operatorname{diam}_{\rho_G}(A)<\infty$.
            By definition of $\rho_G$ we have
            $\operatorname{diam}_{\rho_{\faktor{G}{N}}}(\pi(A))<\infty$,
            so by Proposition~\ref{для конечно порожденных групп} the
            set $\pi(A)$ is finite. Hence $A$ is contained in a
            finite union of cosets of $N$, so $A\in\langle N\rangle$.
            Therefore $\mathfrak{B}_{\rho_G}=\langle N\rangle$.

      \item We show that the quotient map $\pi\colon G\to\faktor{G}{N}$
            is a coarse map. If $B\subset\faktor{G}{N}$ is finite then
            $\pi^{-1}(B)\in\langle N\rangle$, since it is a finite
            union of cosets of $N$. Hence $\pi$ is proper. Bornologicity
            of $\pi$ follows from the description of $\langle N\rangle$
            and the definition of $\rho_G$.

      \item The map $i\colon\faktor{G}{N}\to G$ sends each element of
            the quotient to a chosen representative. It is proper since
            the preimage of any finite set or finite union of cosets
            in $G$ is finite. Bornologicity follows from the definition
            of $\rho_G$.

      \item Since $\pi\circ i=\operatorname{id}$, closeness to the
            identity requires no verification for this composition.
            For $i\circ\pi$: the set
            \[
                \{(i\circ\pi(g))^{-1}g\mid g\in G\}
            \]
            belongs to $\langle N\rangle$, since $i\circ\pi(g)$ lies
            in the same coset as $g$, so $(i\circ\pi(g))^{-1}g\in N$
            for every $g\in G$.\qedhere
  \end{itemize}
\end{proof}

\begin{definition}\label{def-genernormal}
  Metrics defined by formula~\eqref{eq-factormetric} are called
  \emph{metrics realized via the normal subgroup~$N$}.
\end{definition}

We now turn to examples of improper metrics on $\mathbb{Z}$. Every
subgroup of $\mathbb{Z}$ is normal and has the form
\[
    k\mathbb{Z}:=\{kn\mid n\in\mathbb{Z}\}.
\]

\begin{exmpl}
  All metrics on $\mathbb{Z}$ realized via subgroups $k\mathbb{Z}$
  are not coarsely equivalent to any proper left-invariant metric on
  this group.
\end{exmpl}

\begin{proof}
  \begin{itemize}
      \item The bornologies $\langle k\mathbb{Z}\rangle$ are countable
            and, by Corollary~\ref{критерий метризуемости борнологии},
            metrizable. Denote the corresponding metric by $\rho_k$.

      \item Suppose $\rho_k$ is proper. By the uniqueness of the
            proper metric (Theorem~\ref{Smith}) we have
            $(\mathbb{Z},\rho_k)\stackrel{c}{\sim}(\mathbb{Z},d)$,
            where $d(x,y)=|x-y|$.

      \item By Proposition~\ref{реализуемость метрики на подгруппе},
            $(\mathbb{Z},\rho_k)\stackrel{c}{\sim}
            (\faktor{\mathbb{Z}}{k\mathbb{Z}},\mathfrak{B}_{\min})$.
            Hence $(\mathbb{Z},d)\stackrel{c}{\sim}
            (\faktor{\mathbb{Z}}{k\mathbb{Z}},\mathfrak{B}_{\min})$,
            which is impossible: it would imply that $\mathbb{Z}$ is
            bounded with respect to $d$, contradicting properness.

      This contradiction shows that $\rho_k$ is not coarsely
      equivalent to any proper left-invariant metric.\qedhere
  \end{itemize}
\end{proof}

The proof uses no specific properties of $\mathbb{Z}$ beyond finite
generation, infiniteness, and the description of its normal subgroups
of finite index. Thus we have in fact produced a class of improper
left-invariant metrics for all infinite finitely generated groups
admitting a normal subgroup of finite index, and we have proved the
following theorem.

\begin{theorem}\label{th-critnonequiv}
  Let $G$ be an infinite finitely generated group and let $d$ be a
  proper left-invariant metric on $G$. If $N$ is a normal subgroup
  of finite index and $\rho_{\faktor{G}{N}}$ is the word metric on
  the quotient group, then $(G,d)$ is not coarsely equivalent to
  $\bigl(\faktor{G}{N},\rho_{\faktor{G}{N}}\bigr)$.
\end{theorem}

\bigskip

Not every countable bornology is generated by normal subgroups. We
now give an example of a metrizable bornological structure on
$\mathbb{Z}$ that generates a metric $\rho$ not realizable via any
subgroup in the sense of Definition~\ref{def-genernormal}: there is
no subgroup whose quotient is coarsely equivalent to $(\mathbb{Z},\rho)$.

In the group $\mathbb{Z}$ choose the subset
\[
    A:=\{0,10,100,\ldots,10^k,\ldots\}.
\]

\begin{exmpl}
  There is no nontrivial subgroup $N\leq\mathbb{Z}$ such that
  $(\mathbb{Z},\langle A\rangle)\stackrel{c}{\sim}
  (\faktor{\mathbb{Z}}{N},\mathfrak{B}_{\min})$.
\end{exmpl}

\begin{proof}
  \begin{itemize}
      \item The bornology $\langle A\rangle$ is countably generated,
            since $\mathbb{Z}$ is abelian. Its basis consists of the
            sets $kA:=A+A+\cdots+A$ ($k$ summands), all integer
            translates of $A$, and all finite sets. Every bounded set
            in $\langle A\rangle$ is contained in a finite union of
            sets of this form. By
            Corollary~\ref{критерий метризуемости борнологии} we
            obtain a bornological metric $\rho$.

      \item Since every nontrivial subgroup $N\leq\mathbb{Z}$ gives a
            finite quotient $\faktor{\mathbb{Z}}{N}$, it suffices to
            show that $\langle A\rangle\neq 2^{\mathbb{Z}}$.

      \item The bornology equals $2^{\mathbb{Z}}$ if and only if it
            contains the set of even integers.

      \item The set $2\mathbb{N}$ cannot be expressed as a finite
            union of basis sets of $\langle A\rangle$: the set $kA$
            consists of multiples of $10$, and the gaps between
            consecutive elements of $A$ grow, so no finite collection
            of translates and finite sums involving $A$ covers all
            even integers. Hence $2\mathbb{N}\notin\langle A\rangle$,
            as required.\qedhere
  \end{itemize}
\end{proof}

This example generalizes to similarly defined sets. Choose a countable
family of subsets
\[
    A_j:=\{0,\,10+10j,\,(10+10j)^2,\,\ldots,\,(10+10j)^k,\,\ldots\}.
\]

\begin{exmpl}
  In the group $\mathbb{Z}$, the bornology
  $\langle A_j\rangle$, $j\in\mathbb{Z}$, is not realizable via any
  subgroup and generates an improper metric not coarsely equivalent
  to any proper left-invariant metric on $\mathbb{Z}$.
\end{exmpl}

Recall that a coarse structure is called \emph{monogenic} if it is
generated by a single set. Monogenicity is a coarse invariant
(see~\cite[Proposition~2.57]{Roe2003LecturesOC}) and is equivalent
to the coarse space being coarsely equivalent to a geodesic space.

In particular, the space $(\mathbb{Z},d)$ with $d(x,y)=|x-y|$ has
the monogenic bounded coarse structure $\mathcal{E}_d$ generated by
$E=\{(x,y)\mid d(x,y)\leq 1\}$.

\begin{proof}
  \begin{itemize}
      \item Non-realizability via a subgroup follows by the same
            argument as in the previous example: $\mathbb{Z}$ itself
            does not belong to this bornology, so it is not a bounded
            set.

      \item The bornology $\langle A_j\rangle$ generated by the
            entire family $\{A_j\}$ has a countable basis consisting
            of finite sums of the sets $A_i$ (with repetitions) and
            their translates. Hence it is countable and therefore
            metrizable.

      \item As noted above, $\mathcal{E}_d$ is monogenic. On the
            other hand, the left bornological coarse structure
            generated by $\langle A_j\rangle$ is not monogenic,
            since there are infinitely many indices $j$ for which
            $A_j$ cannot be obtained from the other $A_i$ by finitely
            many operations (see the proof of
            Proposition~\ref{связь}, where a basis of the coarse
            structure is constructed explicitly). Hence the metric
            generated by $\langle A_j\rangle$ is not coarsely
            equivalent to any proper left-invariant metric on
            $\mathbb{Z}$.\qedhere
  \end{itemize}
\end{proof}

Thus, even in elementary groups, improper left-invariant metrics can
generate the structure of a non-geodesic space.

\printbibliography

@inproceedings{Roe2003LecturesOC,
   title     = {Lectures on Coarse Geometry},
  author    = {Roe, John},
  publisher = {American Mathematical Society},
  series    = {University lecture series},
  year      =  {2003},
    url = {https://www.ams.org/books/ulect/031/},
}

@inproceedings{Tessera2022CoarseFP,
      title={Coarse fixed point properties}, 
      author={Romain Tessera and Jeroen Winkel},
      year={2022},
      eprint={2212.04900},
      archivePrefix={arXiv},
      primaryClass={math.GR},
      url={https://arxiv.org/abs/2212.04900}, 
}

@misc{brodskiy2006coarsestructuresgroupactions,
      title={Coarse structures and group actions}, 
      author={N. Brodskiy and J. Dydak and A. Mitra},
      year={2006},
      eprint={math/0607568},
      archivePrefix={arXiv},
      primaryClass={math.MG},
      url={https://arxiv.org/abs/math/0607568}, 
}

@Article{smith2006,
  author  = {Smith, J.},
  title   = {On asymptotic dimension of countable abelian groups},
  doi     = {10.1016/j.topol.2005.07.011},
  number  = {12},
  pages   = {2047--2054},
  url     = {https://arxiv.org/abs/math/0504447},
  volume  = {153},
  journal = {Topology Appl.},
  year    = {2006},
}

@misc{ma2021endslargescalegroups,
      title={Ends of large scale groups}, 
      author={Yuankui Ma and Hussain Rashed and Jerzy Dydak},
      year={2021},
      eprint={2109.10280},
      archivePrefix={arXiv},
      primaryClass={math.MG},
      url={https://arxiv.org/abs/2109.10280}, 
}

@misc{artigue2026,
      title={Metric-Independent Expansiveness}, 
      author={Alfonso Artigue and Luis Ferrari},
      year={2026},
      eprint={2603.20978},
      archivePrefix={arXiv},
      primaryClass={math.DS},
      url={https://arxiv.org/abs/2603.20978}, 
}

@article{Chung2021,
  title = {Structure and $K$-theory of $\ell^p$ uniform Roe algebras},
  volume = {15},
  ISSN = {1661-6960},
  url = {http://dx.doi.org/10.4171/JNCG/405},
  DOI = {10.4171/jncg/405},
  number = {2},
  journal = {Journal of Noncommutative Geometry},
  publisher = {European Mathematical Society - EMS - Publishing House GmbH},
  author = {Chung,  Yeong Chyuan and Li,  Kang},
  year = {2021},
  pages = {581–614}
}

@ARTICLE{Chung2024-wc,
  title     = "Quasi-countable inverse semigroups as metric spaces, and the
               uniform Roe algebras of locally finite inverse semigroups",
  author    = "Chung, Yeong Chyuan and Mart{\'\i}nez, Diego and Szak{\'a}cs,
               N{\'o}ra",
  journal   = "Groups Geom. Dyn.",
  publisher = "European Mathematical Society - EMS - Publishing House GmbH",
  volume    =  19,
  number    =  4,
  pages     = "1499--1543",
  month     =  jul,
  year      =  2024
}

@ARTICLE{Bell2008-dz,
  title     = "The asymptotic dimension of a curve graph is finite",
  author    = "Bell, Gregory C and Fujiwara, Koji",
  journal   = "J. Lond. Math. Soc. (2)",
  publisher = "Wiley",
  volume    =  77,
  number    =  1,
  pages     = "33--50",
  month     =  feb,
  year      =  2008,
  copyright = "http://onlinelibrary.wiley.com/termsAndConditions\#vor",
  language  = "en"
}

@ARTICLE{Bell2008-qg-1,
  title     = "Asymptotic dimension",
  author    = "Bell, G and Dranishnikov, A",
  journal   = "Topol. Appl.",
  publisher = "Elsevier BV",
  volume    =  155,
  number    =  12,
  pages     = "1265--1296",
  month     =  jun,
  year      =  2008,
  copyright = "https://www.elsevier.com/open-access/userlicense/1.0/",
  language  = "en"
}

\end{document}